\RequirePackage{plautopatch}
\documentclass{amsart}

\usepackage{amsmath,amssymb,amsthm,mathtools,mathrsfs}
\usepackage{iftex}
\ifPDFTeX
  \usepackage[hidelinks]{hyperref}
\else\ifLuaTeX
  \usepackage[hidelinks]{hyperref}
\else\ifXeTeX
  \usepackage[hidelinks]{hyperref}
\else
  \usepackage[dvipdfmx,hidelinks]{hyperref}
\fi\fi\fi
\usepackage{cleveref}

\numberwithin{equation}{section}

\theoremstyle{plain}
\newtheorem{theorem}{Theorem}[section]
\newtheorem{proposition}[theorem]{Proposition}

\newtheorem{corollary}[theorem]{Corollary}
\theoremstyle{definition}
\newtheorem{definition}[theorem]{Definition}
\theoremstyle{remark}
\newtheorem{remark}[theorem]{Remark}

\crefname{theorem}{theorem}{theorems}
\crefname{proposition}{proposition}{propositions}
\crefname{lemma}{lemma}{lemmas}
\crefname{corollary}{corollary}{corollaries}
\crefname{definition}{definition}{definitions}
\crefname{remark}{remark}{remarks}
\crefname{section}{section}{sections}
\crefname{equation}{equation}{equations}

\DeclareMathOperator{\supp}{supp}
\DeclareMathOperator{\diam}{diam}
\DeclareMathOperator{\dis}{dis}
\DeclareMathOperator{\pr}{pr}
\DeclareMathOperator{\Lip}{Lip}

\title[Fell and Wijsman Topologies for the Box Metric]{Metrization of the Fell Topology and the Wijsman Topology for the Box Metric on Closed Downward Families}
\author{Shigeaki Yokota}
\thanks{Graduate School of Science, The University of Osaka, 1-1 Machikaneyama, Toyonaka, Osaka 560-0043, Japan. Corresponding author: \texttt{yokota.shigeaki.fat@osaka-u.ac.jp}.}
\address{Graduate School of Science, The University of Osaka, 1-1 Machikaneyama, Toyonaka, Osaka 560-0043, Japan}
\email{yokota.shigeaki.fat@osaka-u.ac.jp}
\date{}

\begin{document}

\begin{abstract}
We study nonempty Box-closed families of metric measure spaces that are downward under the Lipschitz order. If a space is below a member in this order, it also belongs to the family. Such families need not be pyramids, because directedness is an additional condition. We define a metric on these families by taking the supremum, over all finite metric measure spaces, of the normalized absolute difference between the infimal Box distances from each finite space to the two families. We divide each difference by the sum of the support cardinality and diameter of the finite test space. The topology induced by this metric agrees on the full domain with the relative Fell topology and with the restriction of the Wijsman topology for the Box metric. On the subspace of pyramids, it induces the weak topology and makes the natural inclusion into the hyperspace of nonempty closed subsets a topological embedding for either hyperspace topology. Pyramids associated with metric measure spaces are dense in the pyramid space, which is the completion of metric measure spaces under the metric pulled back along the map that sends each space to its associated pyramid.
\end{abstract}

\maketitle

\section{Introduction}\label{sec:introduction}

We work with complete separable metric spaces equipped with a full-support Borel probability measure. Gromov's work on metric-measure geometry uses the Box distance to compare such spaces and the concentration topology to study their convergence \cite{gromov2007met}\cite[Chapters~4--6]{shioya2016mmg}. The Box distance compares two spaces through a coupling, a probability measure on their product whose marginals are the given measures. On a subset of large coupling mass, pairwise distances in the two spaces must have small distortion, while the mass outside that subset is discarded; see the exact formula in \Cref{eq:box-coupling}.

Given a coupling $\pi$ of the measures on two mm-spaces, pull each space's family of real-valued $1$-Lipschitz functions back to the product. For measurable real-valued functions $f,g$ on the product, their Ky Fan distance is $\inf\{\varepsilon\geq0\mid\pi(\{|f-g|>\varepsilon\})\leq\varepsilon\}$. We identify functions that agree $\pi$-almost everywhere. Compare the pulled-back families by the Hausdorff distance induced by this metric; the observable distance is the infimum of these comparisons over all such couplings \cite[Definition~2.13 and Theorem~5.2]{nakajima2022coupling}. It induces the concentration topology \cite[Chapters~4--6]{shioya2016mmg}. By Shioya's Corollary~5.8, a sequence of mm-spaces exhibits measure concentration if and only if it converges to a one-point mm-space in observable distance \cite[Corollary~5.8]{shioya2016mmg}. Limits in the concentration topology need not be one-point mm-spaces; the equivalence singles out the case corresponding to measure concentration.

Gromov introduced pyramids and the corresponding compactification in the study of concentration \cite[Definitions~3\(\frac12\).51--3\(\frac12\).52, pp.~212--213]{gromov2007met}. Shioya added the nonemptiness and Box-closedness conditions used in our definition, defined the weak topology, and constructed a compact metric compatible with it \cite[Definition~6.3, §6.5, p.~101, §§6.2--6.4, Theorem~6.22]{shioya2016mmg}. The Lipschitz order is the relation $Y\prec X$ when a measure-preserving $1$-Lipschitz map sends the mm-space $X$ to the mm-space $Y$ \cite[Definition~2.10]{shioya2016mmg}. A pyramid is a nonempty, Box-closed, downward-closed, directed family for this order: every pair of its members has a common upper bound in the family \cite[Definition~6.3]{shioya2016mmg}.

Kazukawa, Nakajima, and Shioya study the weak topology on pyramids by means of hyperspace theory \cite{kazukawa-nakajima-shioya2024topological} and ask whether the pyramid space admits a topological embedding into the hyperspace of closed subsets of the Box metric space of mm-spaces, equipped with either the Fell topology or the Wijsman topology on that metric space \cite[Question~9.5]{kazukawa-nakajima-shioya2024topological}. Since a pyramid is already a nonempty Box-closed subset of the space of mm-spaces, its natural inclusion is a direct way to answer their question. We show that the relative Fell topology and the restriction to nonempty Box-closed families downward for the Lipschitz order of the Wijsman topology on $(\mathcal X,\Box)$ coincide. Their common restriction to pyramids is the weak topology. Related work on pyramids includes Esaki, Kazukawa, and Mitsuishi's study of invariants, scaling, and product operations \cite{esaki-kazukawa-mitsuishi2024invariants}.

We identify mm-spaces up to measure-preserving isometry and write $\mathcal X$ for the resulting set of classes. A subfamily $\mathcal A$ is downward if $X\in\mathcal A$ and $Y\prec X$ imply $Y\in\mathcal A$. Write $\mathcal D$ for the set of nonempty Box-closed downward subfamilies of $\mathcal X$. For $X\in\mathcal X$ and $\mathcal A\in\mathcal D$, put
\[
 \Box(X,\mathcal A)\coloneqq\inf_{Y\in\mathcal A}\Box(X,Y).
\]
For a finite mm-space $T$, write $\#T$ for the cardinality of its support and $\diam T$ for its diameter. For nonempty Box-closed downward families $\mathcal A,\mathcal B\in\mathcal D$, define
\begin{equation}\label{eq:intro-dfin}
 d_{\mathrm{fin}}(\mathcal A,\mathcal B)
 \coloneqq
 \sup_{T\text{ finite}}
 \frac{|\Box(T,\mathcal A)-\Box(T,\mathcal B)|}{\#T+\diam T},
 \end{equation}
where the supremum ranges over finite mm-spaces $T$.

For each finite mm-space $T$, consider the map $\mathcal A\mapsto\Box(T,\mathcal A)$ on $\mathcal D$. The initial topology induced by these maps is generated by the sets
\[
 \{\mathcal A\in\mathcal D\mid\Box(T,\mathcal A)\in O\},
\]
where $T$ ranges over all finite metric measure spaces and $O\subset[0,1]$ is open.

The ambient Wijsman topology on the hyperspace of nonempty closed subsets of $(\mathcal X,\Box)$ is generated by point-to-set Box-distance coordinates. Its restriction to $\mathcal D$ is generated by the sets
\[
 \{\mathcal A\in\mathcal D\mid\Box(X,\mathcal A)\in O\},
\]
where $X$ ranges over all of $\mathcal X$ and $O\subset[0,1]$ is open \cite[Chapter~2, §2.1]{beer1993topologies}. The ambient Fell topology is defined on the same hyperspace by open-hit and compact-miss sets. Its restriction to $\mathcal D$ has subbasis consisting of the sets $\{\mathcal A\in\mathcal D\mid\mathcal A\cap U\ne\varnothing\}$ and $\{\mathcal A\in\mathcal D\mid\mathcal A\cap K=\varnothing\}$, where $U$ ranges over open subsets and $K$ over compact subsets of $(\mathcal X,\Box)$ \cite[§5.1, Definition~5.1.1]{beer1993topologies}.

Weak convergence of pyramids is sequential Painlev\'e--Kuratowski convergence, also called weak Hausdorff convergence in mm-space theory. For pyramids $\mathcal P_n$ and $\mathcal P$, $\mathcal P_n$ converges weakly to $\mathcal P$ as $n\to\infty$ if and only if
\[
 X\in\mathcal P\ \Longrightarrow\ \lim_{n\to\infty}\Box(X,\mathcal P_n)=0,
 \qquad
 X\notin\mathcal P\ \Longrightarrow\ \liminf_{n\to\infty}\Box(X,\mathcal P_n)>0
\]
for every $X\in\mathcal X$ \cite[Definition~6.4]{shioya2016mmg}.

\begin{theorem}\label{thm:intro-metric}
The function $d_{\mathrm{fin}}$ is a metric on $\mathcal D$. The initial topology just described coincides on $\mathcal D$ with the topology induced by $d_{\mathrm{fin}}$. Its restriction to the pyramid space is the weak topology.
\end{theorem}

\begin{theorem}\label{thm:intro-hyperspace}
On $\mathcal D$, the relative Fell topology, the restriction of the Wijsman topology on $(\mathcal X,\Box)$, and the topology induced by $d_{\mathrm{fin}}$ coincide.
\end{theorem}

\begin{corollary}\label{cor:intro-hyperspace-embedding}
The inclusion of the space of pyramids into the hyperspace of nonempty closed subsets of $(\mathcal X,\Box)$ is a topological embedding for the ambient Fell and Wijsman topologies.
\end{corollary}

For $X\in\mathcal X$, write $\mathcal P_X\coloneqq\{Y\in\mathcal X\mid Y\prec X\}$ for the pyramid associated with $X$. We use the distance between associated pyramids as a metric on mm-spaces:
\[
 d_{\mathrm{fin}}(X,Y)\coloneqq d_{\mathrm{fin}}(\mathcal P_X,\mathcal P_Y).
\]

\begin{proposition}\label{prop:intro-associated-completion}
The map $X\mapsto\mathcal P_X$ embeds the pullback metric space of mm-spaces isometrically into the pyramid space equipped with $d_{\mathrm{fin}}$. Its image is dense by \cite[Lemma~7.14]{shioya2016mmg}. Thus this pyramid space is its completion.
\end{proposition}

For each finite mm-space $T$, a member of a nonempty Box-closed downward family that nearly minimizes its distance from $T$ can be replaced by a dominated mm-space in the compact class determined by $T$, with distance from $T$ no larger up to an arbitrarily small error. Downward closure keeps these replacements in the family, and compactness together with closedness gives a member attaining the minimum. This attainment turns inequalities for $\Box(T,\mathcal A)$ into open-hit and compact-miss conditions, which describe the Fell topology. For the metric, finite nets in compact classes of finite mm-spaces with bounded cardinality and diameter control the coordinate differences there. Since Box distances lie in $[0,1]$, the denominator $\#T+\diam T$ makes the normalized differences uniformly small outside those bounds.

Writing $d_{\mathrm{conc}}$ for the observable distance, Shioya's concentration characterization and \Cref{thm:intro-metric} show that, for every sequence $X_n$ and every mm-space $X$, $d_{\mathrm{conc}}(X_n,X)\to0$ as $n\to\infty$ if and only if $d_{\mathrm{fin}}(\mathcal P_{X_n},\mathcal P_X)\to0$ as $n\to\infty$ \cite[Proposition~6.13]{shioya2016mmg}.

\Cref{sec:preliminaries} fixes the metric-measure, coupling, order, Box, and pyramid conventions. \Cref{sec:finite-metric} proves that $d_{\mathrm{fin}}$ is a metric, identifies its topology on $\mathcal D$ with that generated by distances from finite mm-spaces, and shows that its restriction to pyramids is the weak topology. \Cref{sec:hyperspace} identifies the relative Fell topology and the restriction to $\mathcal D$ of the Wijsman topology on $(\mathcal X,\Box)$, proves that $\mathcal D$ is closed in the ambient hyperspace of nonempty closed subsets of $(\mathcal X,\Box)$, and relates net convergence and the pyramid completion to weak convergence and its known compactness and density results. \Cref{sec:concentration} proves the normalized forward estimate from observable distance and records the absence of a uniform converse modulus.

\section{Preliminaries}\label{sec:preliminaries}

This section records the definitions and standard facts used in the paper. We follow \cite[Chapters~1--6]{shioya2016mmg} for metric measure spaces, the Box distance, the Lipschitz order, pyramids, and weak convergence. The coupling and order facts supply the inputs for the realization in Section~\ref{sec:finite-metric}. Compactness of probability measures on the cube fixed by a finite mm-space gives attainment, while compactness of finite-mm-space classes with bounded cardinality and diameter provides finite nets to control the differences between their distances to two families.

\begin{definition}[Metric measure space {\cite[Definition~2.8 and §2.2]{shioya2016mmg}}]\label{def:mm-space}
A triple $(X,d_X,\mu_X)$, or simply $X$, is an \emph{mm-space} if $(X,d_X)$ is a complete separable metric space and $\mu_X$ is a Borel probability measure whose support is $X$.
\end{definition}

We identify mm-spaces up to measure-preserving isometry and denote the set of mm-isomorphism classes by $\mathcal X$. The distance and measure of an mm-space $X$ are denoted by $d_X$ and $\mu_X$. For probability measures $\mu$ and $\nu$, let $\mathcal T(\mu,\nu)$ be the set of their couplings. We use $\mathcal T$ rather than the customary $\Pi(\mu,\nu)$ to avoid confusion with the pyramid space $\Pi$.

For mm-spaces $X$ and $Y$, a coupling $\pi\in\mathcal T(\mu_X,\mu_Y)$, and a nonempty set $S\subset X\times Y$, set
\[
 \dis S\coloneqq
 \sup\{|d_X(x,x')-d_Y(y,y')|\mid (x,y),(x',y')\in S\},
\]
and put $\dis\varnothing\coloneqq\infty$. The Box distance has the closed-good-set representation
\begin{equation}\label{eq:box-coupling}
 \Box(X,Y)=
 \min_{\substack{\pi\in\mathcal T(\mu_X,\mu_Y)\\ S\subset X\times Y\ \mathrm{closed}}}
 \max\{\dis S,1-\pi(S)\}.
\end{equation}
The representation and attainment in \eqref{eq:box-coupling} are proved in \cite[Theorem~1.1]{nakajima2022coupling}.

\begin{definition}[Lipschitz order {\cite[Definition~2.10]{shioya2016mmg}}]\label{def:lipschitz-order}
For mm-spaces $X$ and $Y$, write $Y\prec X$ if there is a measure-preserving $1$-Lipschitz map from $X$ to $Y$.
\end{definition}

The Lipschitz order is antisymmetric on $\mathcal X$ \cite[Proposition~2.11]{shioya2016mmg}. We use the additive-error extension in \cite[Definition~1.2]{nakajima2024lipschitz}. For a closed set $S\subset X\times Y$, put
\[
 \dis_{\succ}S\coloneqq
 \sup\{d_Y(y,y')-d_X(x,x')\mid (x,y),(x',y')\in S\},
 \qquad
 \dis_{\succ}\varnothing\coloneqq0.
\]
For $\varepsilon\geq0$, write $X\succ_\varepsilon Y$ if there are $\pi\in\mathcal T(\mu_X,\mu_Y)$ and a closed set $S\subset X\times Y$ such that
\[
 \dis_{\succ}S\leq\varepsilon,
 \qquad
 \pi(S)\geq1-\varepsilon.
\]
The zero-error and transitivity properties of this relation are proved in \cite[Theorem~1.3]{nakajima2024lipschitz}. Define
\[
 \Box_{\succ}(X,Y)\coloneqq\inf\{\varepsilon\geq0\mid X\succ_\varepsilon Y\}.
\]
By \cite[Corollary~1.6]{nakajima2024lipschitz},
\begin{align}
 \Box_{\succ}(X,Y)=0&\quad\Longleftrightarrow\quad Y\prec X,\label{eq:succ-zero}\\
 \Box_{\succ}(X,Z)&\leq\Box_{\succ}(X,Y)+\Box_{\succ}(Y,Z).\label{eq:succ-triangle}
\end{align}
The closed-good-set definitions also give
\begin{equation}\label{eq:succ-box}
 \Box_{\succ}(X,Y)\leq\Box(X,Y).
\end{equation}
If $X_n\to X$ and $Y_n\to Y$ in the Box distance, $\varepsilon_n\to\varepsilon$, and $X_n\succ_{\varepsilon_n}Y_n$ for every $n$, then $X\succ_\varepsilon Y$ by \cite[Theorem~1.5]{nakajima2024lipschitz}.

\begin{definition}[Pyramid {\cite[Definition~6.3]{shioya2016mmg}}]\label{def:pyramid}
A nonempty subset $\mathcal P\subset\mathcal X$ is a \emph{pyramid} if it is Box-closed, downward for $\prec$, and directed. The last condition means that for every $X,Y\in\mathcal P$, there is $Z\in\mathcal P$ such that $X\prec Z$ and $Y\prec Z$.
\end{definition}

Denote the space of all pyramids by $\Pi$. Weak convergence of pyramids is sequential Painlev\'e--Kuratowski convergence, also called weak Hausdorff convergence in mm-space theory \cite[Definition~6.4]{shioya2016mmg}. Thus $\mathcal P_n$ converges weakly to $\mathcal P$ if
\begin{align*}
 X\in\mathcal P&\quad\Longrightarrow\quad \Box(X,\mathcal P_n)\longrightarrow0,\\
 X\notin\mathcal P&\quad\Longrightarrow\quad \liminf_{n\to\infty}\Box(X,\mathcal P_n)>0
\end{align*}
as $n\to\infty$.

Finite mm-spaces are Box-dense in $\mathcal X$ \cite[Proposition~4.20]{shioya2016mmg}. For a positive integer $N$ and $D\geq0$, the finite mm-spaces with at most $N$ points and diameter at most $D$ form a Box-compact set \cite[Theorem~4.27]{shioya2016mmg}. The space $\mathcal X$ is separable \cite[Proposition~4.25]{shioya2016mmg}. For Borel probability measures, weak convergence means convergence of integrals against every bounded continuous real-valued function \cite[Definition~1.11]{shioya2016mmg}. For $A\subset M$ and $\varepsilon>0$, write $A^\varepsilon\coloneqq\{x\in M\mid d(x,A)<\varepsilon\}$. The Prohorov distance is the infimum of $\varepsilon>0$ for which $\mu(A)\leq\nu(A^\varepsilon)+\varepsilon$ and $\nu(A)\leq\mu(A^\varepsilon)+\varepsilon$ for every Borel set $A\subset M$ \cite[Definition~1.14]{shioya2016mmg}. On a compact metric space, the Borel probability measures are compact for weak convergence, which agrees with convergence in the Prohorov distance \cite[Lemma~1.17(1) and (3)]{shioya2016mmg}. If $\mu$ and $\nu$ are probability measures on one complete separable metric space $(M,d)$, then
\begin{equation}\label{eq:fixed-base-box}
 \Box((M,d,\mu),(M,d,\nu))\leq2d_P(\mu,\nu)
\end{equation}
by \cite[Proposition~4.12]{shioya2016mmg}. Here $d_P$ denotes the Prohorov distance. In this display, each triple $(M,d,\mu)$ or $(M,d,\nu)$ denotes the support of its measure equipped with the restricted metric, so both triples satisfy the full-support convention for mm-spaces.

\section{Box distances from finite mm-spaces}\label{sec:finite-metric}

For an mm-space $X$ and a finite mm-space $T$, $\Box_{\succ}(X,T)$ is the infimum of $\varepsilon\geq0$ such that a coupling of $X$ and $T$ assigns mass at least $1-\varepsilon$ to a closed set $S\subset X\times T$ on which $d_T(t,t')\leq d_X(x,x')+\varepsilon$ for all $(x,t),(x',t')\in S$.
Write the support points of $T$ as $t_1,\ldots,t_m$, and put $D=\diam T$. Whenever $0<\varepsilon<1$ and $\Box_{\succ}(X,T)<\varepsilon$, choose the corresponding coupling $\pi$ and nonempty closed good set $S$ from the definition above. For $i=1,\ldots,m$, define
\[
 u_i(x)\coloneqq\min\left\{D,\inf_{(x',t')\in S}\bigl(d_X(x,x')+d_T(t',t_i)\bigr)\right\}.
\]
For fixed $(x',t')$, the expression inside the infimum changes by at most $d_X(x,y)$ when $x$ changes to $y$. Taking the infimum over $S$ and then the minimum with $D$ preserves this bound, so each $u_i$ is $1$-Lipschitz. Write $u=(u_1,\ldots,u_m)\colon X\to[0,D]^m$. For $(x,t)\in S$, using $(x,t)$ in the infimum gives $u_i(x)\leq d_T(t,t_i)$; the one-sided inequality on $S$ and the triangle inequality give $d_T(t,t_i)-u_i(x)\leq\varepsilon$. Thus $d_T(t,t_i)-u_i(x)\in[0,\varepsilon]$ for all $(x,t)\in S$ and $1\leq i\leq m$. For two points of $S$, the difference between their errors in each coordinate has absolute value at most $\varepsilon$. Since $a(t)\coloneqq(d_T(t,t_i))_{i=1}^m$ is isometric for the supremum metric, this gives $\bigl|\|u(x)-u(y)\|_\infty-d_T(t,s)\bigr|\leq\varepsilon$ for $(x,t),(y,s)\in S$. Push $\pi$ forward by $(x,t)\mapsto(u(x),t)$. The closure of the image of $S$ in $[0,D]^m\times T$ has pushforward mass at least $1-\varepsilon$ and retains the distance bound. The pushforward of $\mu_X$ by $u$, restricted to its support, is an mm-space in this fixed cube dominated by $X$. For $\mathcal A\in\mathcal D$, \Cref{eq:succ-box} gives $\Box_{\succ}(X,T)\leq\Box(T,X)$. Applying the construction to members of $\mathcal A$ whose distances from $T$ are arbitrarily close to $\Box(T,\mathcal A)$ yields dominated replacements in $\mathcal A$ when this distance is below $1$; the one-point space handles the value $1$. Box-compactness of the support-restricted probability measures on $[0,D]^m$ and Box-closedness of $\mathcal A$ then give attainment.

Fix a finite mm-space $T$ with support $\{t_1,\ldots,t_m\}$, and put $D\coloneqq\diam T$. Define
\[
 \mathscr C_T\coloneqq
 \{([0,D]^m,\|\cdot\|_\infty,\nu)\mid \nu\text{ is a Borel probability measure on }[0,D]^m\}.
\]
Each member is restricted to the support of its measure. This convention also applies when $D=0$. By \eqref{eq:fixed-base-box} and compactness of the probability measures on $[0,D]^m$, the class $\mathscr C_T$ is Box-compact.

\begin{theorem}\label{thm:finite-realization}
Let $X$ be an mm-space and let $0<\varepsilon<1$. If
\[
 \Box_{\succ}(X,T)<\varepsilon,
\]
then there is an mm-space $Z\in\mathscr C_T$ such that $Z\prec X$ and $\Box(T,Z)\leq\varepsilon$.
\end{theorem}

\begin{proof}
Choose $\pi\in\mathcal T(\mu_X,\mu_T)$ and a closed set $S\subset X\times T$ such that
\begin{equation}\label{eq:realization-good}
 \pi(S)>1-\varepsilon,
 \qquad
 d_T(t,t')<d_X(x,x')+\varepsilon
\end{equation}
for all $(x,t),(x',t')\in S$. Since $\varepsilon<1$, the set $S$ is nonempty.

For $i=1,\ldots,m$, define
\begin{equation}\label{eq:ui-definition}
 u_i(x)\coloneqq
 \min\left\{D,\inf_{(x',t')\in S}\bigl(d_X(x,x')+d_T(t',t_i)\bigr)\right\},
\end{equation}
and let $u\coloneqq(u_1,\ldots,u_m)\colon X\to[0,D]^m$. The triangle inequality and the $1$-Lipschitz property of truncation to $[0,D]$ give
\begin{equation}\label{eq:ui-lipschitz}
 |u_i(x)-u_i(y)|\leq d_X(x,y)
\end{equation}
for all $x,y\in X$. Thus $u$ is $1$-Lipschitz for the maximum metric.

Fix $(x,t)\in S$. Taking $(x',t')=(x,t)$ in \eqref{eq:ui-definition} gives $u_i(x)\leq d_T(t,t_i)$. For every $(x',t')\in S$, \eqref{eq:realization-good} and the triangle inequality give
\[
 d_T(t,t_i)<d_X(x,x')+d_T(t',t_i)+\varepsilon.
\]
Taking the infimum over $(x',t')\in S$ gives
\[
 d_T(t,t_i)-\varepsilon\leq\inf_{(x',t')\in S}\bigl(d_X(x,x')+d_T(t',t_i)\bigr).
\]
Since $d_T(t,t_i)\leq D$, replacing this infimum by its minimum with $D$ preserves the lower bound. Therefore,
\begin{equation*}
 d_T(t,t_i)-\varepsilon\leq u_i(x)\leq d_T(t,t_i).
 \end{equation*}

Set $a(t)\coloneqq(d_T(t,t_i))_{i=1}^m$. The coordinate indexed by a point $t_j$ shows that
\[
 \|a(t)-a(t')\|_\infty=d_T(t,t').
\]
For every $i$, the two coordinate errors are as follows:
\[
\begin{aligned}
 d_T(t,t_i)-u_i(x)&\in[0,\varepsilon],\\
 d_T(t',t_i)-u_i(x')&\in[0,\varepsilon].
\end{aligned}
\]
\[
\begin{aligned}
&\left|(u_i(x)-u_i(x'))-(d_T(t,t_i)-d_T(t',t_i))\right|\\
&=\left|(d_T(t,t_i)-u_i(x))-(d_T(t',t_i)-u_i(x'))\right|\\
&\leq\varepsilon.
\end{aligned}
\]
Hence
\begin{equation}\label{eq:constant-one-distortion}
 \left|\|u(x)-u(x')\|_\infty-d_T(t,t')\right|\leq\varepsilon
\end{equation}
for all $(x,t),(x',t')\in S$.

Let $\nu\coloneqq u_*\mu_X$ and let $Z\coloneqq([0,D]^m,\|\cdot\|_\infty,\nu)$ restricted to $\supp\nu$. Since $u$ is continuous and $\mu_X$ has full support, $u(X)\subset\supp\nu$. In fact, if $O$ is an open neighborhood of $u(x)$, then $u^{-1}(O)$ is an open neighborhood of $x$, and
\[
 \nu(O)=\mu_X(u^{-1}(O))>0.
\]
It follows from \eqref{eq:ui-lipschitz} that $u\colon X\to\supp\nu$ is a measure-preserving $1$-Lipschitz map. Thus $Z\prec X$, and $Z\in\mathscr C_T$.

The pushforward
\[
 \beta\coloneqq(t,u(x))_*\pi
\]
is a coupling of $\mu_T$ and $\nu$ on $T\times\supp\nu$. Let $G$ be the closure in $T\times\supp\nu$ of the image of $S$ under $(x,t)\mapsto(t,u(x))$. Then $G$ is closed and $\beta(G)\geq\pi(S)>1-\varepsilon$. By \eqref{eq:constant-one-distortion} and continuity of the two metrics, $\dis G\leq\varepsilon$. Applying \eqref{eq:box-coupling} to the full coupling $\beta$ and the closed set $G$ gives $\Box(T,Z)\leq\varepsilon$. This completes the proof.
\end{proof}

\begin{proposition}\label{prop:finite-coordinate}
For every mm-space $X$ and every finite mm-space $T$,
\begin{equation}\label{eq:associated-coordinate}
 \Box_{\succ}(X,T)=\Box(T,\mathcal P_X).
\end{equation}
\end{proposition}

\begin{proof}
If $Z\in\mathcal P_X$, then \eqref{eq:succ-zero}, \eqref{eq:succ-triangle}, and \eqref{eq:succ-box} give
\[
 \Box_{\succ}(X,T)\leq\Box_{\succ}(X,Z)+\Box_{\succ}(Z,T)\leq\Box(Z,T).
\]
Taking the infimum over $Z\in\mathcal P_X$ proves one inequality. If $\Box_{\succ}(X,T)<1$, apply \Cref{thm:finite-realization} to every $\varepsilon$ satisfying
\[
 \Box_{\succ}(X,T)<\varepsilon<1.
\]
The resulting space $Z\in\mathcal P_X$ satisfies $\Box(T,Z)\leq\varepsilon$. Letting $\varepsilon\downarrow\Box_{\succ}(X,T)$ gives the reverse inequality. If $\Box_{\succ}(X,T)=1$, the first inequality and the bound $\Box(T,\mathcal P_X)\leq1$ give equality. This completes the proof.
\end{proof}

\begin{theorem}\label{thm:compact-attainment}
Let $\mathcal A\subset\mathcal X$ be a nonempty Box-closed downward family. For every finite mm-space $T$,
\begin{equation}\label{eq:compact-attainment}
 \Box(T,\mathcal A)=\min_{Z\in\mathcal A\cap\mathscr C_T}\Box(T,Z).
\end{equation}
\end{theorem}

\begin{proof}
Put $r\coloneqq\Box(T,\mathcal A)$. The one-point mm-space belongs to $\mathcal A$ and has a realization in $\mathscr C_T$. Thus $\mathcal A\cap\mathscr C_T$ is nonempty. If $r=1$, this one-point space realizes the minimum because the Box distance is at most $1$.

Assume that $r<1$. Choose $X_j\in\mathcal A$ and numbers $\varepsilon_j<1$ such that
\[
 \Box(T,X_j)\longrightarrow r,
 \qquad
 \Box_{\succ}(X_j,T)<\varepsilon_j,
 \qquad
 \varepsilon_j\longrightarrow r
\]
as $j\to\infty$. By \Cref{thm:finite-realization}, there is $Z_j\in\mathscr C_T$ such that $Z_j\prec X_j$ and $\Box(T,Z_j)\leq\varepsilon_j$. Since $\mathcal A$ is downward, $Z_j\in\mathcal A$. Compactness of $\mathscr C_T$ gives a subsequence converging in the Box distance to some $Z\in\mathscr C_T$. The closedness of $\mathcal A$ gives $Z\in\mathcal A$, and
\[
 r\leq\Box(T,Z)=\lim_{j\to\infty}\Box(T,Z_j)
 \leq\lim_{j\to\infty}\varepsilon_j=r.
\]
Thus $Z$ realizes the minimum in \eqref{eq:compact-attainment}. This completes the proof.
\end{proof}

\begin{theorem}\label{thm:coordinates-weak}
For pyramids $\mathcal P_n$ and $\mathcal P$, the following conditions are equivalent.
\begin{enumerate}
 \item The sequence $\mathcal P_n$ converges weakly to $\mathcal P$ as $n\to\infty$.
 \item For every $X\in\mathcal X$, one has $\Box(X,\mathcal P_n)\to\Box(X,\mathcal P)$ as $n\to\infty$.
 \item For every finite mm-space $T$, one has $\Box(T,\mathcal P_n)\to\Box(T,\mathcal P)$ as $n\to\infty$.
\end{enumerate}
\end{theorem}

\begin{proof}
Assume (1), and fix $X\in\mathcal X$. For every $Z\in\mathcal P$, weak convergence gives $\Box(Z,\mathcal P_n)\to0$ as $n\to\infty$. Therefore,
\begin{equation}\label{eq:coordinate-upper}
 \limsup_{n\to\infty}\Box(X,\mathcal P_n)\leq\Box(X,\mathcal P).
\end{equation}

We prove the reverse bound first when $X=T$ is finite. By \Cref{thm:compact-attainment}, choose
\[
 Z_n\in\mathcal P_n\cap\mathscr C_T,
 \qquad
 \Box(T,Z_n)=\Box(T,\mathcal P_n).
\]
If the lower bound
\[
 \Box(T,\mathcal P)\leq\liminf_{n\to\infty}\Box(T,\mathcal P_n)
\]
failed, there would be a subsequence $(n_\ell)$ such that
\[
 \lim_{\ell\to\infty}\Box(T,\mathcal P_{n_\ell})=\liminf_{n\to\infty}\Box(T,\mathcal P_n)<\Box(T,\mathcal P).
\]
Compactness of $\mathscr C_T$ gives a further subsequence $(n_{\ell_r})$ such that $Z_{n_{\ell_r}}\to Z$ in the Box distance as $r\to\infty$. Since
\[
 \Box(Z,\mathcal P_{n_{\ell_r}})\leq\Box(Z,Z_{n_{\ell_r}})\longrightarrow0\quad\text{as }r\to\infty,
\]
the exclusion condition in weak convergence implies $Z\in\mathcal P$. It follows that
\[
 \Box(T,\mathcal P)\leq\Box(T,Z)=\lim_{r\to\infty}\Box(T,Z_{n_{\ell_r}})
 =\lim_{r\to\infty}\Box(T,\mathcal P_{n_{\ell_r}})=\liminf_{n\to\infty}\Box(T,\mathcal P_n),
\]
contradicting the choice of $(n_\ell)$. Thus the lower bound holds.
Together with \eqref{eq:coordinate-upper}, this proves (2) for finite $X$. Finite Box-density and the uniform $1$-Lipschitz bound
\[
|\Box(X,\mathcal A)-\Box(Y,\mathcal A)|\leq\Box(X,Y)
\]
for nonempty $\mathcal A\subset\mathcal X$ give, for every finite mm-space $T$,
\[
 |\Box(X,\mathcal P_n)-\Box(X,\mathcal P)|\leq 2\Box(X,T)+|\Box(T,\mathcal P_n)-\Box(T,\mathcal P)|.
\]
Choose finite $T$ with $\Box(X,T)$ arbitrarily small, then let $n\to\infty$. This extends the convergence to every $X\in\mathcal X$.

Condition (2) implies (1) by the definition of weak convergence. If $X\in\mathcal P$, the limiting coordinate is zero. If $X\notin\mathcal P$, closedness gives $\Box(X,\mathcal P)>0$. Condition (2) immediately implies (3), and finite Box-density with the same uniform $1$-Lipschitz bound shows that (3) implies (2). This completes the proof.
\end{proof}

\begin{definition}\label{def:finite-domination-metric}
For nonempty Box-closed downward families $\mathcal A,\mathcal B\subset\mathcal X$, we use the metric $d_{\mathrm{fin}}$ defined in \eqref{eq:intro-dfin}.
\end{definition}

Nakajima--Shioya compactify unmeasured metric-space pyramids by considering, for each positive integer $k$, finite metric spaces with at most $k$ points and diameter at most $k$. They compare these classes using the Hausdorff distances induced by the Gromov--Hausdorff distance and aggregate those distances with summable weights. For each finite mm-space $T$, $d_{\mathrm{fin}}$ instead compares $|\Box(T,\mathcal A)-\Box(T,\mathcal B)|$ after division by $\#T+\diam T$ \cite{nakajima-shioya2024compactification}.

For each positive integer $k$, the compact class of finite mm-spaces $T$ with at most $k$ points and diameter at most $k$ has a finite net that uniformly controls the differences between $\Box(T,\mathcal A)$ and $\Box(T,\mathcal B)$ on this class. For finite mm-spaces $T$ outside this class, the normalized absolute difference $|\Box(T,\mathcal A)-\Box(T,\mathcal B)|/(\#T+\diam T)$ satisfies
\begin{equation}\label{eq:metric-tail}
 \sup\left\{
 \frac{|\Box(T,\mathcal A)-\Box(T,\mathcal B)|}{\#T+\diam T}
 \mathrel{}\middle|\mathrel{}T\text{ finite},\ \#T>k\text{ or }\diam T>k
 \right\}\leq\frac1k.
\end{equation}

\begin{proposition}\label{prop:metric-topology}
On $\mathcal D$, $d_{\mathrm{fin}}$ is a metric and induces the initial topology defined in \Cref{sec:introduction}. On $\Pi$, this topology is the weak topology.
\end{proposition}

\begin{proof}
The absolute difference of two real-valued coordinates is symmetric and satisfies the triangle inequality; dividing by $\#T+\diam T$ and taking the supremum over finite $T$ gives these properties for $d_{\mathrm{fin}}$. If $d_{\mathrm{fin}}(\mathcal A,\mathcal B)=0$, then $\Box(T,\mathcal A)=\Box(T,\mathcal B)$ for every finite $T$. For any $X\in\mathcal X$, choose finite $T_j$ with $\Box(X,T_j)\to0$ as $j\to\infty$. The $1$-Lipschitz bound for distance coordinates gives
\[
 \begin{aligned}
 |\Box(X,\mathcal A)-\Box(X,\mathcal B)|&\leq2\Box(X,T_j)+|\Box(T_j,\mathcal A)-\Box(T_j,\mathcal B)|\\
 &=2\Box(X,T_j)\longrightarrow0\quad\text{as }j\to\infty.
 \end{aligned}
\]
Thus the distance functions of $\mathcal A$ and $\mathcal B$ agree on all of $\mathcal X$. A closed set is the zero set of its distance function, and hence $\mathcal A=\mathcal B$.

We compare the two topologies without restricting to sequences. Fix $\mathcal A$ and $\eta>0$. Choose a positive integer $k$ such that $1/k<\eta$. The class of finite mm-spaces with at most $k$ points and diameter at most $k$ is compact, so take a finite $\eta$-net $\{T_1,\ldots,T_N\}$ in this class. If a nonempty Box-closed downward family $\mathcal B$ satisfies
\[
 |\Box(T_j,\mathcal A)-\Box(T_j,\mathcal B)|<\eta
 \qquad (j=1,\ldots,N),
\]
then, for every finite $T$ with $\#T\leq k$ and $\diam T\leq k$, choose $T_j$ with $\Box(T,T_j)\leq\eta$. For each fixed $\mathcal A$, the map $T\mapsto\Box(T,\mathcal A)$ is $1$-Lipschitz for the Box distance, so
\[
 \begin{aligned}
 |\Box(T,\mathcal A)-\Box(T,\mathcal B)|
 &\leq|\Box(T,\mathcal A)-\Box(T_j,\mathcal A)|+|\Box(T_j,\mathcal A)-\Box(T_j,\mathcal B)|\\
 &\qquad+|\Box(T_j,\mathcal B)-\Box(T,\mathcal B)|\\
 &\leq2\Box(T,T_j)+|\Box(T_j,\mathcal A)-\Box(T_j,\mathcal B)|<3\eta.
 \end{aligned}
\]
The strict margin is uniform over the finite net. Since $\#T+\diam T\geq1$ and the normalized contribution from finite mm-spaces $T$ with $\#T>k$ or $\diam T>k$ is at most $1/k<\eta$ by \eqref{eq:metric-tail}, the neighborhood defined by these finitely many coordinates is contained in the $d_{\mathrm{fin}}$-ball of radius $3\eta$. Conversely, for each finite $T$,
\[
 |\Box(T,\mathcal A)-\Box(T,\mathcal B)|
 \leq(\#T+\diam T)d_{\mathrm{fin}}(\mathcal A,\mathcal B).
\]
Thus these topologies coincide.

The weak topology on $\Pi$ is metrizable by \cite[Theorem~6.22]{shioya2016mmg}. Both it and the topology of $d_{\mathrm{fin}}$ are therefore sequential. By \Cref{thm:coordinates-weak}, their convergent sequences coincide, so the two topologies coincide. This completes the proof.
\end{proof}

\section{Hyperspace topology and compactification}\label{sec:hyperspace}

The Fell topology records whether a closed set hits an open set and misses a compact set, while the Wijsman topology on $(\mathcal X,\Box)$ records point-to-set Box distances. For a finite mm-space $T$, the class $\mathscr C_T$ from Section~\ref{sec:finite-metric} consists of the mm-spaces defined by probability measures on the cube $[0,\diam T]^{\#T}$ with its supremum metric, restricted to their supports. For $\mathcal A\in\mathcal D$ and $r\geq0$, \Cref{thm:compact-attainment} says that $\Box(T,\mathcal A)$ is attained in $\mathcal A\cap\mathscr C_T$. Thus $\Box(T,\mathcal A)>r$ is equivalent to missing the compact set $K_{T,r}\coloneqq\mathscr C_T\cap\{Z\in\mathcal X\mid\Box(T,Z)\leq r\}$, while $\Box(T,\mathcal A)<r$ is equivalent to hitting the open ball $B_{\Box}(T,r)\coloneqq\{Z\in\mathcal X\mid\Box(T,Z)<r\}$. These conditions make every finite-mm-space distance coordinate Fell-continuous; Box-density of finite mm-spaces extends this to all point-to-set Box coordinates and gives the topology comparison. The same open-hit and compact-miss construction separates a nonempty closed family that fails downward closure, proving that $\mathcal D$ is closed in the ambient hyperspace of nonempty closed subsets of $(\mathcal X,\Box)$. On pyramids, the resulting topology is the weak topology, and the known compactness of the weak pyramid space and density of associated pyramids give the completion statement.

\begin{definition}[Fell topology {\cite[§5.1, Definition~5.1.1]{beer1993topologies}}]
Let $(M,d)$ be a metric space, and set $F(M)\coloneqq\{A\subset M\mid A\text{ is closed}\}$, including the empty set, and $F_+(M)\coloneqq F(M)\setminus\{\varnothing\}$. The \emph{Fell topology} on $F(M)$ is generated by the subbasis
\[
 \{A\in F(M)\mid A\cap U\neq\varnothing\},\qquad
 \{A\in F(M)\mid A\cap K=\varnothing\},
\]
where $U\subset M$ is open and $K\subset M$ is compact.
\end{definition}

\begin{definition}[Wijsman topology {\cite[Chapter~2, §2.1]{beer1993topologies}}]
For $x\in M$ and $A\in F_+(M)$, set $d(x,A)\coloneqq\inf_{a\in A}d(x,a)$. The \emph{Wijsman topology} on $F_+(M)$ is the initial topology of the maps $A\mapsto d(x,A)$, one for each $x\in M$.
\end{definition}

\begin{definition}[Restriction of the Wijsman topology]
The Wijsman topology on the metric space $(\mathcal X,\Box)$, restricted to $\mathcal D$, is the initial topology of the maps $f_X\colon\mathcal D\to\mathbb R$ given by
\[
 f_X(\mathcal A)\coloneqq\Box(X,\mathcal A)=\inf_{Y\in\mathcal A}\Box(X,Y),\qquad X\in\mathcal X.
\]
\end{definition}

\begin{theorem}\label{thm:hyperspace-topologies}
On the nonempty Box-closed downward families, the relative Fell topology, the restriction to this domain of the Wijsman topology on $(\mathcal X,\Box)$, and the topology induced by $d_{\mathrm{fin}}$ coincide.
\end{theorem}

\begin{proof}
Fix a finite mm-space $T$ and $r\geq0$, and use $B_{\Box}(T,r)$ and $K_{T,r}$ as defined at the start of this section.
The set $K_{T,r}$ is compact. By \Cref{thm:compact-attainment}, every nonempty Box-closed downward family $\mathcal A$ satisfies
\begin{align*}
 \Box(T,\mathcal A)>r&\quad\Longleftrightarrow\quad \mathcal A\cap K_{T,r}=\varnothing,\\
 \Box(T,\mathcal A)<r&\quad\Longleftrightarrow\quad \mathcal A\cap B_{\Box}(T,r)\neq\varnothing.
\end{align*}
The right-hand sides are respectively a compact-miss condition and an open-hit condition. Thus, for every finite mm-space $T$, the map $\mathcal A\mapsto\Box(T,\mathcal A)$ is continuous in the relative Fell topology.

For any metric space $(M,d)$ and nonempty closed $A\subset M$, let $A'$ range over the nonempty closed subsets of $M$. The Fell subbasis consists of open-hit sets and compact-miss sets. If $x\in A\cap U$ for an open $U\subset M$, choose $r>0$ with $B(x,r)\subset U$. The Wijsman-open condition $d(x,A')<r$ contains $A$ and forces $A'\cap U\ne\varnothing$. If $K\subset M$ is compact and disjoint from $A$, then $\delta\coloneqq\inf_{x\in K}d(x,A)>0$; the case $K=\varnothing$ is automatic. Choose a finite $\delta/4$-net $\{x_1,\ldots,x_m\}$ in $K$. The finite intersection of Wijsman-open coordinate conditions $d(x_j,A')>\delta/2$ contains $A$. If some $y\in A'\cap K$ existed, a net point $x_j$ with $d(x_j,y)<\delta/4$ would give $d(x_j,A')<\delta/4$, a contradiction. Thus this finite intersection forces $A'\cap K=\varnothing$. Every Fell subbasic neighborhood is Wijsman-open, so the Fell topology is weaker than the Wijsman topology.

For an arbitrary $X\in\mathcal X$, choose finite mm-spaces $T_j$ such that $\Box(T_j,X)\to0$ as $j\to\infty$. For every nonempty $\mathcal A\subset\mathcal X$,
\begin{equation}\label{eq:uniform-coordinate-density}
 |\Box(T_j,\mathcal A)-\Box(X,\mathcal A)|\leq\Box(T_j,X).
\end{equation}
The bound is independent of $\mathcal A$. Hence, for every $X\in\mathcal X$, the map $\mathcal A\mapsto\Box(X,\mathcal A)$ is a uniform limit on $\mathcal D$ of Fell-continuous maps $\mathcal A\mapsto\Box(T_j,\mathcal A)$ and is Fell-continuous. The restriction to $\mathcal D$ of the Wijsman topology on $(\mathcal X,\Box)$ is therefore weaker than the relative Fell topology. The subbasis argument above gives the reverse inclusion, so these two relative topologies coincide.

By \eqref{eq:uniform-coordinate-density}, the maps $\mathcal A\mapsto\Box(X,\mathcal A)$ for all $X\in\mathcal X$ and the maps $\mathcal A\mapsto\Box(T,\mathcal A)$ for finite mm-spaces $T$ induce the same topology. This topology is induced by $d_{\mathrm{fin}}$ by \Cref{prop:metric-topology}. This completes the proof.
\end{proof}

\begin{corollary}\label{cor:hyperspace-embedding}
The inclusion of $\Pi$ into $F_+(\mathcal X)$ is a topological embedding for the ambient Fell and Wijsman topologies.
\end{corollary}

\begin{proof}
By Theorem~\ref{thm:hyperspace-topologies} and Proposition~\ref{prop:metric-topology}, both relative hyperspace topologies on $\Pi$ coincide with the weak topology of pyramids. This completes the proof.
\end{proof}

The next theorem concerns this ambient hyperspace; its proof uses only downward closure, not directedness.

\begin{theorem}\label{thm:downward-fell-closed}
The subset $\mathcal D$ is closed in $F_+(\mathcal X)$ for both the ambient Fell topology and the ambient Wijsman topology on $(\mathcal X,\Box)$.
\end{theorem}

\begin{proof}
We separate each $\mathcal A\in F_+(\mathcal X)\setminus\mathcal D$ from $\mathcal D$ by an ambient Fell-open neighborhood. Since $\mathcal A$ is not downward, choose $X\in\mathcal A$ and $Y\prec X$ with $Y\notin\mathcal A$. Closedness of $\mathcal A$ gives
\[
 \delta\coloneqq\Box(Y,\mathcal A)>0.
\]
Choose a finite mm-space $T$ such that $e\coloneqq\Box(T,Y)<\delta/3$. For each fixed $\mathcal A$, the map $T\mapsto\Box(T,\mathcal A)$ is $1$-Lipschitz for the Box distance, so
\[
 b\coloneqq\Box(T,\mathcal A)\geq\delta-e>e.
\]
Choose $r$ with $e<r<b$, and put
\[
 K\coloneqq\mathscr C_T\cap\{Z\in\mathcal X\mid\Box(T,Z)\leq r\},
 \qquad U\coloneqq\{Z\in\mathcal X\mid\Box(X,Z)<r-e\}.
\]
The set $K$ is compact. Since $\Box(T,\mathcal A)=b>r$, we have $\mathcal A\cap K=\varnothing$. Also $X\in\mathcal A\cap U$, so $\mathcal A$ belongs to the Fell-open set of families that hit $U$ and miss $K$.

Let $\mathcal B\in\mathcal D$ hit $U$, and choose $W\in\mathcal B\cap U$. Applying the oriented triangle inequality from \Cref{eq:succ-triangle} to $W,X,Y,T$, using $\Box_{\succ}(W,X)\leq\Box(W,X)$ from \Cref{eq:succ-box} and $\Box_{\succ}(X,Y)=0$ from $Y\prec X$ and \Cref{eq:succ-zero}, and using $\Box_{\succ}(Y,T)\leq\Box(Y,T)=e$, gives
\[
 \Box_{\succ}(W,T)
 \leq\Box_{\succ}(W,X)+\Box_{\succ}(X,Y)+\Box_{\succ}(Y,T)
 \leq\Box(W,X)+e<r.
\]
Because $\mathcal B$ is downward and contains $W$, it contains $\mathcal P_W$. By \Cref{prop:finite-coordinate},
\[
 \Box(T,\mathcal B)\leq\Box(T,\mathcal P_W)=\Box_{\succ}(W,T)<r.
\]
By \Cref{thm:compact-attainment}, applied to $\mathcal B$ and $T$, there is a point of $\mathcal B\cap\mathscr C_T$ at distance less than $r$ from $T$. Thus $\mathcal B\cap K\ne\varnothing$. No member of $\mathcal D$ can therefore hit $U$ and miss $K$, so the indicated Fell-open neighborhood separates $\mathcal A$ from $\mathcal D$.

Every ambient Fell-open set is open in the ambient Wijsman topology on $(\mathcal X,\Box)$, by the general Fell--Wijsman inclusion for metric spaces established above. Hence $\mathcal D$ is also closed in that Wijsman topology. This completes the proof.
\end{proof}

Let $\Lambda$ be a directed set, meaning that every finite set of indices has a common upper bound, and let $(A_\lambda)_{\lambda\in\Lambda}$ be a net in $F(\mathcal X)$. Define $\operatorname{Li}_{\lambda\in\Lambda}A_\lambda$ to consist of the points $x$ such that every neighborhood $U$ of $x$ is eventually hit by $A_\lambda$. Define $\operatorname{Ls}_{\lambda\in\Lambda}A_\lambda$ by requiring every neighborhood $U$ of $x$ to be hit cofinally.
Here “eventually” means that some $\lambda_0\in\Lambda$ satisfies $A_\lambda\cap U\ne\varnothing$ for all $\lambda\geq\lambda_0$. “Cofinally” means that for every $\lambda_0\in\Lambda$ there is $\lambda\geq\lambda_0$ with $A_\lambda\cap U\ne\varnothing$.

\begin{definition}[Painlev\'e--Kuratowski convergence of nets {\cite[Definition~8.1]{kazukawa-nakajima-shioya2024topological}}]\label{def:kp-convergence}
The net $(A_\lambda)_{\lambda\in\Lambda}$ Painlev\'e--Kuratowski converges to $A\in F(\mathcal X)$ if
\[
 \operatorname{Li}_{\lambda\in\Lambda}A_\lambda
 =\operatorname{Ls}_{\lambda\in\Lambda}A_\lambda=A.
\]
\end{definition}

\begin{corollary}\label{cor:downward-kp}
Let $(\mathcal A_\lambda)_{\lambda\in\Lambda}$ be a net in $\mathcal D$, and let $\mathcal A\in\mathcal D$. The following are equivalent: $\mathcal A_\lambda$ Painlev\'e--Kuratowski converges to $\mathcal A$; $\mathcal A_\lambda$ converges to $\mathcal A$ in the ambient Fell topology; it converges in the ambient Wijsman topology on $(\mathcal X,\Box)$; and $d_{\mathrm{fin}}(\mathcal A_\lambda,\mathcal A)\to0$.
\end{corollary}

\begin{proof}
By \Cref{thm:hyperspace-topologies}, convergence in the relative Fell topology, the restriction of the Wijsman topology on $(\mathcal X,\Box)$, and the $d_{\mathrm{fin}}$ topology are equivalent. It remains to identify Fell convergence with Painlev\'e--Kuratowski convergence. Suppose first that $\mathcal A_\lambda$ converges in the relative Fell topology. Then its Box-distance coordinates converge by the Wijsman-topology equality. If $X\in\mathcal A$, the limit coordinate is zero, so every neighborhood of $X$ is eventually hit. If $X\notin\mathcal A$, closedness gives $\delta\coloneqq\Box(X,\mathcal A)>0$. Coordinate convergence implies that eventually $\Box(X,\mathcal A_\lambda)>\delta/2$, so the neighborhood $\{Z\in\mathcal X\mid\Box(X,Z)<\delta/2\}$ is eventually missed. Thus the lower and upper Painlev\'e--Kuratowski limits both equal $\mathcal A$.

Conversely, suppose these two limits equal $\mathcal A$. Every open-hit condition at a point of $\mathcal A$ is then eventually satisfied. Let $K\subset\mathcal X$ be compact and disjoint from $\mathcal A$. For each $x\in K$, membership $x\notin\operatorname{Ls}_{\lambda\in\Lambda}\mathcal A_\lambda$ gives a neighborhood $U_x$ of $x$ that is eventually missed. Choose a finite subcover of $K$ by such neighborhoods. Directedness of $\Lambda$ gives one index after which all sets $U_x$ in this finite subcover are missed, so $K$ is eventually missed. This proves Fell convergence. This completes the proof.
\end{proof}

Kazukawa, Nakajima, and Shioya define $\tau_K$ on $F(\mathcal X)$ as the finest topology whose convergence is weaker than Painlev\'e--Kuratowski convergence. In particular, every Painlev\'e--Kuratowski-convergent net is $\tau_K$-convergent to the same limit. They also show that $\tau_F\subset\tau_K$ \cite[Section~8]{kazukawa-nakajima-shioya2024topological}.

\begin{corollary}\label{cor:downward-tau-k-closed}
The subset $\mathcal D$ is closed in $F(\mathcal X)$ equipped with $\tau_K$.
\end{corollary}

\begin{proof}
By \Cref{thm:downward-fell-closed}, $\mathcal D$ is closed in $F_+(\mathcal X)$ for the relative Fell topology. It remains to consider the empty set. Let $*$ be the one-point mm-space. Every nonempty downward family contains $*$, so the compact-miss Fell-open set $\{\mathcal A\in F(\mathcal X)\mid\mathcal A\cap\{*\}=\varnothing\}$ is a neighborhood of $\varnothing$ disjoint from $\mathcal D$. Thus $\mathcal D$ is closed in the full Fell topology. Since $\tau_F\subset\tau_K$, it is closed in $\tau_K$.

This completes the proof.
\end{proof}

\begin{corollary}\label{cor:downward-tau-k-fell}
The subspace topology on $\mathcal D$ induced by $\tau_K$ coincides with the relative Fell topology.
\end{corollary}

\begin{proof}
Let $V$ be $\tau_K$-open in $\mathcal D$, and let a net in $\mathcal D$ converge in the relative Fell topology to $\mathcal A\in V$. By \Cref{cor:downward-kp}, it Painlev\'e--Kuratowski converges to $\mathcal A$. By Section~8 of \cite{kazukawa-nakajima-shioya2024topological}, it then converges in $\tau_K$, and hence is eventually in $V$. The net criterion for openness shows that $V$ is relatively Fell-open. Conversely, $\tau_F\subset\tau_K$ implies that every relatively Fell-open subset of $\mathcal D$ is $\tau_K$-open in $\mathcal D$. Thus the two subspace topologies coincide.

This completes the proof.
\end{proof}

For mm-spaces $X$ and $Y$, the pullback metric is their associated pyramids' distance:
\begin{equation}\label{eq:mm-pullback}
 d_{\mathrm{fin}}(X,Y)
 \coloneqq d_{\mathrm{fin}}(\mathcal P_X,\mathcal P_Y).
\end{equation}

\begin{proposition}\label{prop:associated-completion}
The pyramids associated with mm-spaces are dense in $(\Pi,d_{\mathrm{fin}})$ \cite[Lemma~7.14]{shioya2016mmg}. The map $X\mapsto\mathcal P_X$ is an isometric embedding of $(\mathcal X,d_{\mathrm{fin}})$ into $(\Pi,d_{\mathrm{fin}})$, and the latter space is its completion.
\end{proposition}

\begin{proof}
Shioya's compactness theorem makes $\Pi$ compact in the weak topology \cite[Theorem~6.22]{shioya2016mmg}. By \Cref{prop:metric-topology}, this topology is induced by $d_{\mathrm{fin}}$, so $\Pi$ is compact and complete for that metric. Shioya's density theorem gives weak density of the pyramids associated with mm-spaces \cite[Lemma~7.14]{shioya2016mmg}; hence they are also dense for $d_{\mathrm{fin}}$.

The map is isometric by \eqref{eq:mm-pullback}. If $d_{\mathrm{fin}}(X,Y)=0$, then $\mathcal P_X=\mathcal P_Y$. Thus $X\prec Y$ and $Y\prec X$, and antisymmetry of the Lipschitz order gives $X=Y$. The map is an embedding, and its dense image in the complete space $\Pi$ proves the completion statement. This completes the proof.
\end{proof}

\begin{corollary}\label{cor:associated-concentration}
For mm-spaces $X_n$ and $X$,
\[
 d_{\mathrm{fin}}(X_n,X)\longrightarrow0
 \quad\Longleftrightarrow\quad
 d_{\mathrm{conc}}(X_n,X)\longrightarrow0
\]
as $n\to\infty$.
\end{corollary}

\begin{proof}
By \Cref{prop:metric-topology}, convergence on the left is equivalent to weak convergence of $\mathcal P_{X_n}$ to $\mathcal P_X$ as $n\to\infty$. By \cite[Proposition~6.13]{shioya2016mmg}, this weak convergence is equivalent to $d_{\mathrm{conc}}(X_n,X)\to0$ as $n\to\infty$. This completes the proof.
\end{proof}

\section{Quantitative comparison with concentration}\label{sec:concentration}

Let $X$ and $Y$ be source mm-spaces and $T$ a finite target mm-space. By \Cref{prop:finite-coordinate}, $\Box_{\succ}(X,T)=\Box(T,\mathcal P_X)$ and $\Box_{\succ}(Y,T)=\Box(T,\mathcal P_Y)$. Thus their absolute difference divided by $\#T+\diam T$ is the $T$-term in the supremum defining $d_{\mathrm{fin}}(X,Y)$. The realization in Section~\ref{sec:finite-metric} constructs one $1$-Lipschitz function on $X$ for each support point of $T$ from a coupling of $X$ and $T$ that nearly realizes $\Box_{\succ}(X,T)$. For any $h>d_{\mathrm{conc}}(X,Y)$, a single coupling of $X$ and $Y$ transfers all these functions simultaneously to $1$-Lipschitz functions on $Y$, each within $h$ outside a set of mass at most $h$. Gluing this coupling with the same $X$--$T$ coupling over their common $X$-marginal produces a measure on $X\times Y\times T$ with the chosen couplings as its marginals. The union bound adds at most $\#T h$ to the exceptional mass. For two retained triples, the transferred function values add at most $2h$ to the distance error. When $\#T\geq2$, this is at most $\#T h$; the singleton case is handled in the proof. These estimates control $|\Box_{\succ}(X,T)-\Box_{\succ}(Y,T)|$ by $\#T h$, up to the arbitrarily small error in choosing the near-optimal $X$--$T$ coupling. After division by $\#T+\diam T$, the $T$-term is therefore at most $h$ up to that arbitrarily small error. Taking the supremum over $T$ and letting $h\downarrow d_{\mathrm{conc}}(X,Y)$ yields $d_{\mathrm{fin}}(X,Y)\leq d_{\mathrm{conc}}(X,Y)$.

For a metric space $M$, write
\[
 \Lip_1(M)\coloneqq
 \{f\colon M\to\mathbb R\mid |f(x)-f(x')|\leq d_M(x,x')\text{ for all }x,x'\in M\}.
\]
For real-valued measurable functions $f,g$ on a probability space with measure $\mu$, identified when they agree $\mu$-almost everywhere, the \emph{Ky Fan metric} \cite[Definition~2.13]{nakajima2022coupling} is
\[
 d_{\mathrm{KF}}^\mu(f,g)
 \coloneqq
 \inf\{\varepsilon\geq0\mid \mu(\{|f-g|>\varepsilon\})\leq\varepsilon\}.
\]
We use the coupling representation of the observable distance from \cite[Theorem~5.2]{nakajima2022coupling}: $d_{\mathrm{conc}}(X,Y)$ is the infimum over $\pi\in\mathcal T(\mu_X,\mu_Y)$ of the Hausdorff distance, induced by $d_{\mathrm{KF}}^\pi$, between the pullbacks of $\Lip_1(X)$ and $\Lip_1(Y)$ to $X\times Y$.

\begin{proposition}\label{prop:forward-finite}
For every finite mm-space $T$,
\begin{equation}\label{eq:forward-finite}
 |\Box_{\succ}(X,T)-\Box_{\succ}(Y,T)|
 \leq(\#T)d_{\mathrm{conc}}(X,Y).
\end{equation}
Consequently,
\[
 d_{\mathrm{fin}}(X,Y)\leq d_{\mathrm{conc}}(X,Y).
\]
\end{proposition}

\begin{proof}
We first prove $\Box_{\succ}(Y,T)-\Box_{\succ}(X,T)\leq m\,d_{\mathrm{conc}}(X,Y)$. The case $\Box_{\succ}(X,T)=1$ follows from $\Box_{\succ}(Y,T)\leq1$. Otherwise, we transfer to $Y$ the selected source functions $u_1,\ldots,u_m$, one for each support point of $T$, using a common coupling of $X$ and $Y$.

If $\#T=1$, both one-sided errors vanish. Assume that $m\coloneqq\#T\geq2$, and fix $h>d_{\mathrm{conc}}(X,Y)$. By \cite[Theorem~5.2]{nakajima2022coupling}, choose $\pi^{XY}\in\mathcal T(\mu_X,\mu_Y)$ such that the Hausdorff Ky Fan distance between the two pulled-back $1$-Lipschitz function families is less than $h$.

Fix
\[
 \Box_{\succ}(X,T)<\varepsilon<1.
\]
In the proof of \Cref{thm:finite-realization}, choose $\pi^{XT}\in\mathcal T(\mu_X,\mu_T)$, a closed set $S\subset X\times T$, and functions $u_1,\ldots,u_m\in\Lip_1(X)$ such that $\pi^{XT}(S)>1-\varepsilon$ and
\begin{equation}\label{eq:forward-source}
 d_T(t,t')
 \leq\max_{1\leq i\leq m}|u_i(x)-u_i(x')|+\varepsilon
\end{equation}
for all $(x,t),(x',t')\in S$. For every $i$, choose $v_i\in\Lip_1(Y)$ such that
\[
 d_{\mathrm{KF}}^{\pi^{XY}}(u_i\circ\pr_1,v_i\circ\pr_2)<h,
\]
where $\pr_1$ and $\pr_2$ are the coordinate projections from $X\times Y$ to $X$ and $Y$, respectively.
The definition of the Ky Fan distance then gives
\[
 \pi^{XY}\bigl(\{(x,y)\mid |u_i(x)-v_i(y)|>h\}\bigr)\leq h.
\]

The spaces $X,Y,T$ are complete separable metric spaces, and $\pi^{XT}$ and $\pi^{XY}$ are Borel probability couplings with common $X$-marginal $\mu_X$. Applying the disintegration theorem \cite[Theorem~1.8(1)--(3)]{shioya2016mmg} to the projections $\pr_X\colon X\times T\to X$ and $\pr_X\colon X\times Y\to X$ gives conditional probability measures on $X\times T$ and $X\times Y$, respectively. By part (2), these measures are supported on the fibers $\{x\}\times T$ and $\{x\}\times Y$ for $\mu_X$-almost every $x$. Push them forward by $\pr_T$ and $\pr_Y$ to obtain $\kappa_x^T$ and $\kappa_x^Y$. For each $x$, these are probability measures on $T$ and $Y$, respectively, and $x\mapsto\kappa_x^T(C)$ and $x\mapsto\kappa_x^Y(B)$ are Borel for every Borel $C\subset T$ and $B\subset Y$. Part (3) gives, for every Borel $A\subset X$, $B\subset Y$, and $C\subset T$,
\[
 \pi^{XT}(A\times C)=\int_A\kappa_x^T(C)\,d\mu_X(x),\qquad
 \pi^{XY}(A\times B)=\int_A\kappa_x^Y(B)\,d\mu_X(x).
\]
Thus
\[
 \pi^{XT}(dx,dt)=\mu_X(dx)\kappa_x^T(dt),\qquad
 \pi^{XY}(dx,dy)=\mu_X(dx)\kappa_x^Y(dy).
\]
Before applying the $\pi$--$\lambda$ theorem, recall that a $\pi$-system is a family of sets closed under finite intersections. A Dynkin system containing a $\pi$-system that generates a $\sigma$-algebra contains that generated $\sigma$-algebra.

For each $x$, put $\lambda_x\coloneqq\kappa_x^Y\otimes\kappa_x^T$, a probability measure on $Y\times T$. For Borel $B\subset Y$ and $C\subset T$,
\[
 \lambda_x(B\times C)=\kappa_x^Y(B)\kappa_x^T(C),
\]
which is Borel measurable in $x$. Define
\[
 \mathscr L_1\coloneqq\{F\subset Y\times T\mid F\text{ is Borel and }x\mapsto\lambda_x(F)\text{ is Borel measurable}\}.
\]
This is a Dynkin system. It contains $Y\times T$ because each $\lambda_x$ is a probability measure. If $F\in\mathscr L_1$, then $\lambda_x((Y\times T)\setminus F)=1-\lambda_x(F)$, so the complement also belongs to $\mathscr L_1$. If $(F_j)_{j\geq1}$ is a pairwise disjoint sequence in $\mathscr L_1$, then countable additivity and monotone convergence give
\[
 \lambda_x\left(\bigcup_{j\geq1}F_j\right)=\lim_{n\to\infty}\sum_{j=1}^n\lambda_x(F_j),
\]
which is Borel measurable as a pointwise limit of Borel functions. Thus $\mathscr L_1$ is closed under countable disjoint unions. The measurable rectangles $B\times C$ form a $\pi$-system that generates the product Borel $\sigma$-algebra, since $Y$ and $T$ are separable metric spaces. The rectangle identity shows that each such rectangle belongs to $\mathscr L_1$. The $\pi$-$\lambda$ theorem now shows that $x\mapsto\lambda_x(F)$ is Borel measurable for every Borel $F\subset Y\times T$.

For a Borel set $D\subset X\times Y\times T$, its section
\[
D_x\coloneqq\{(y,t)\in Y\times T\mid(x,y,t)\in D\}
\]
is Borel. Define
\[
 \mathscr L_2\coloneqq\{D\subset X\times Y\times T\mid D\text{ is Borel and }x\mapsto\lambda_x(D_x)\text{ is Borel measurable}\}.
\]
This is a Dynkin system. It contains $X\times Y\times T$ because $\lambda_x(Y\times T)=1$. For $D\in\mathscr L_2$, the section of its complement is $(Y\times T)\setminus D_x$, so $\lambda_x((X\times Y\times T\setminus D)_x)=1-\lambda_x(D_x)$, and the complement belongs to $\mathscr L_2$. For a pairwise disjoint sequence $(D_j)_{j\geq1}$ in $\mathscr L_2$, the sections $(D_j)_x$ are pairwise disjoint and
\[
 \lambda_x\left(\left(\bigcup_{j\geq1}D_j\right)_x\right)=\lim_{n\to\infty}\sum_{j=1}^n\lambda_x((D_j)_x)
\]
by countable additivity and monotone convergence. This is a pointwise limit of Borel functions, so $\mathscr L_2$ is closed under countable disjoint unions. For a rectangle $D=A\times F$ with Borel $A\subset X$ and $F\subset Y\times T$, the function $x\mapsto\lambda_x(D_x)$ equals $\lambda_x(F)$ on $A$ and $0$ outside $A$. It is Borel measurable by the first extension. The rectangles $A\times F$ form a generating $\pi$-system for the product Borel $\sigma$-algebra. The $\pi$-$\lambda$ theorem therefore shows that $x\mapsto\lambda_x(D_x)$ is Borel measurable for every Borel $D\subset X\times Y\times T$. Define the measure on the full product Borel $\sigma$-algebra by
\[
 \Gamma(D)\coloneqq\int_X\lambda_x(D_x)\,d\mu_X(x).
\]
Countable additivity follows from monotone convergence, and $\Gamma(X\times Y\times T)=1$ because each $\lambda_x$ is a probability measure. In kernel notation this is $\Gamma(dx,dy,dt)=\mu_X(dx)\kappa_x^Y(dy)\kappa_x^T(dt)$. For Borel sets $A\subset X$, $B\subset Y$, and $C\subset T$, its rectangle formula is
\[
 \Gamma(A\times B\times C)=\int_A\kappa_x^Y(B)\kappa_x^T(C)\,d\mu_X(x).
\]
Integrating out the other kernel gives
\[
 \begin{aligned}
 \Gamma_{X,T}(dx,dt)&=\mu_X(dx)\kappa_x^T(dt)=\pi^{XT}(dx,dt),\\
 \Gamma_{X,Y}(dx,dy)&=\mu_X(dx)\kappa_x^Y(dy)=\pi^{XY}(dx,dy).
 \end{aligned}
\]

Let $E$ be the intersection of the inverse image of $S$ with the $m$ conditions $|u_i(x)-v_i(y)|\leq h$. The union bound gives
\[
 \begin{aligned}
 \Gamma(E^c)&\leq\Gamma(\{(x,y,t)\mid(x,t)\notin S\})+\sum_{i=1}^m\Gamma(\{(x,y,t)\mid|u_i(x)-v_i(y)|>h\})\\
 &<\varepsilon+mh,\qquad\text{so}\qquad \Gamma(E)>1-\varepsilon-mh.
 \end{aligned}
\]
For two points $(x,y,t),(x',y',t')\in E$ and each $i$,
\[
 \begin{aligned}
 |u_i(x)-u_i(x')|&\leq|u_i(x)-v_i(y)|+|v_i(y)-v_i(y')|+|v_i(y')-u_i(x')|\\
 &\leq|v_i(y)-v_i(y')|+2h.
 \end{aligned}
\]
Therefore, \eqref{eq:forward-source} gives
\[
 d_T(t,t')\leq\max_i|v_i(y)-v_i(y')|+\varepsilon+2h\leq d_Y(y,y')+\varepsilon+2h.
\]
Let $R$ be the image of $E$ under the projection $(x,y,t)\mapsto(y,t)$, and let $G$ be its closure in $Y\times T$. The $Y\times T$-marginal of $\Gamma$ assigns mass at least $\Gamma(E)$ to $G$. Choosing preimages in $E$ shows that $d_T(t,t')-d_Y(y,y')\leq\varepsilon+2h$ for every two points $(y,t),(y',t')\in R$. Continuity of the metrics gives the same bound on $G$, so $\dis_{\succ}G\leq\varepsilon+2h$. Since $m\geq2$,
\[
 \Box_{\succ}(Y,T)
 \leq\max\{\varepsilon+2h,\varepsilon+mh\}
 =\varepsilon+mh.
\]
Letting first $\varepsilon\downarrow\Box_{\succ}(X,T)$ and then $h\downarrow d_{\mathrm{conc}}(X,Y)$ yields
\[
 \Box_{\succ}(Y,T)-\Box_{\succ}(X,T)
 \leq m d_{\mathrm{conc}}(X,Y).
\]
Interchanging $X$ and $Y$ proves \eqref{eq:forward-finite}.

By \eqref{eq:associated-coordinate} and \eqref{eq:mm-pullback},
\[
 \begin{aligned}
 d_{\mathrm{fin}}(X,Y)&=\sup_{T\text{ finite}}\frac{|\Box_{\succ}(X,T)-\Box_{\succ}(Y,T)|}{\#T+\diam T}\\
 &\leq\sup_{T\text{ finite}}\frac{\#T\,d_{\mathrm{conc}}(X,Y)}{\#T+\diam T}\leq d_{\mathrm{conc}}(X,Y).
 \end{aligned}
\]
This completes the proof.
\end{proof}

\begin{remark}\label{rem:no-uniform-converse}
There is no single function $\omega\colon[0,1]\to[0,\infty)$ that works simultaneously for all mm-space pairs, satisfies $\omega(t)\to0$ as $t\downarrow0$, and obeys
\[
 d_{\mathrm{conc}}(X,Y)\leq\omega\bigl(d_{\mathrm{fin}}(X,Y)\bigr)
\]
for all mm-spaces $X$ and $Y$. For $n\geq1$, let $X_n\coloneqq S^n(\sqrt n)$ be the $n$-sphere of radius $\sqrt n$ with the Euclidean distance induced from $\mathbb R^{n+1}$ and normalized Riemannian volume measure. The pyramids associated with these spheres converge weakly by \cite[Theorem~7.40]{shioya2016mmg}, and the sequence has no $d_{\mathrm{conc}}$-Cauchy subsequence by \cite[Remark~6.26]{shioya2016mmg}. Since the present metric induces the weak topology of pyramids, $d_{\mathrm{fin}}(X_n,X_m)\to0$ as $n,m\to\infty$, and any such $\omega$ would imply $d_{\mathrm{conc}}(X_n,X_m)\to0$, a contradiction.
The fixed-target equivalence in \Cref{cor:associated-concentration} concerns convergence to one fixed mm-space, whereas this inequality would have to hold uniformly for all pairs.\end{remark}

\section*{Acknowledgments}
The author used OpenAI models in preparing this manuscript. The author reviewed and revised the mathematical content and takes full responsibility for the final manuscript.

\section*{Statements and Declarations}
\subsection*{Funding}
The author was supported by the Japan Society for the Promotion of Science (JSPS) KAKENHI Grant Number JP24K00523.

\end{document}